\documentclass[letterpaper, 10 pt, conference]{ieeeconf}  
\IEEEoverridecommandlockouts                              
\usepackage{graphicx}
\usepackage{url}
\usepackage{ifthen}
\usepackage{cite}
\usepackage{amsmath}
\usepackage{amsfonts}

\usepackage{amsthm}
\usepackage{comment}

\def\real{{\mathchoice%
{\hbox{\rm\setbox1=\hbox{I}\copy1\kern-.45\wd1 R}}
{\hbox{\rm\setbox1=\hbox{I}\copy1\kern-.45\wd1 R}}
{\hbox{\scriptsize\rm\setbox1=\hbox{I}\copy1\kern-.45\wd1 R}}
{\hbox{\scriptsize\rm\setbox1=\hbox{I}\copy1\kern-.45\wd1 R}}}}
\def\Zint{{\mathchoice{\setbox1=\hbox{\sf Z}\copy1\kern-.75\wd1\box1}
{\setbox1=\hbox{\sf Z}\copy1\kern-.75\wd1\box1}
{\setbox1=\hbox{\scriptsize\sf Z}\copy1\kern-.75\wd1\box1}
{\setbox1=\hbox{\scriptsize\sf Z}\copy1\kern-.75\wd1\box1}}}
\newcommand{\complex}{ \hbox{\rm C\kern-0.45em\rule[.07em]{.02em}{.58em}%
\kern 0.43em}}

\newcommand{\be}{\begin{equation}}                         
\newcommand{\ee}{\end{equation}}
\newcommand{\beqr}{\begin{eqnarray}}
\newcommand{\eeqr}{\end{eqnarray}}
\newcommand{\beqrx}{\begin{eqnarray*}}
\newcommand{\eeqrx}{\end{eqnarray*}}
\newcommand{\ba}{\left[ \begin{array}}
\newcommand{\ea}{\\ \end{array} \right]}
\newcommand{\ei}{\end{itemize}}

\newtheorem{theorem}{Theorem}

\newtheorem{corollary}[theorem]{Corollary}
\newtheorem{lemma}[theorem]{Lemma}

\newcommand{\E}[1]{{\mathbb E}\left[ #1 \right]}

\DeclareMathOperator*{\argmin}{arg\,min}

\renewcommand{\baselinestretch}{0.954}

\title{\LARGE \bf
A Stochastic Interpretation of Stochastic Mirror Descent:\\ Risk-Sensitive Optimality
}

\author{Navid Azizan and Babak Hassibi
\thanks{This work was supported in part by the National Science Foundation under grants CCF-1423663, CCF-1409204 and ECCS-1509977, by a grant from Qualcomm Inc., by NASA's Jet Propulsion Laboratory through the President and Director's Fund, and by an Amazon (AWS) AI Fellowship.}
\thanks{N. Azizan is with the Department of Computing and Mathematical Sciences,
        California Institute of Technology, Pasadena, CA 91125, USA
        {\tt\small azizan@caltech.edu}}%
\thanks{B. Hassibi is with the Department of Electrical Engineering, California Institute of Technology, Pasadena, CA 91125, USA
        {\tt\small hassibi@caltech.edu}}%
}

\begin{document}

\maketitle
\thispagestyle{empty}
\pagestyle{empty}

\begin{abstract}
Stochastic mirror descent (SMD) is a fairly new family of algorithms that has recently found a wide range of applications in optimization, machine learning, and control. It can be considered a generalization of the classical stochastic gradient algorithm (SGD), where instead of updating the weight vector along the negative direction of the stochastic gradient, 
the update is performed in a ``mirror domain'' defined by the gradient of a (strictly convex) potential function.
This potential function, and the mirror domain it yields, provides considerable flexibility in the algorithm compared to SGD. While many properties of SMD have already been obtained in the literature, in this paper we exhibit a new interpretation of SMD, namely that it is a risk-sensitive optimal estimator when the unknown weight vector and additive noise are non-Gaussian and belong to the exponential family of distributions. The analysis also suggests a modified version of SMD, which we refer to as symmetric SMD (SSMD). The proofs rely on some simple properties of Bregman divergence, which allow us to extend results from quadratics and Gaussians to certain convex functions and exponential families in a rather seamless way.
\end{abstract}

\section{Introduction}

Stochastic mirror descent (SMD) has become one of the most widely used families of algorithms for optimization, machine learning, and beyond \cite{nemirovski1983problem, beck2003mirror,cesa2012mirror,zhou2017stochastic,nedic2014stochastic,azizan2019characterization,raginsky2012continuous}, which includes the popular stochastic gradient descent (SGD) as a special case. The convergence behavior of such algorithms have been extensively studied in the literature \cite{nemirovski2009robust,nesterov2009primal}, under various assumptions.
Several other properties and interpretations of SMD have recently been proven in the literature\cite{xiao2010dual,gunasekar2018characterizing}. In earlier work, we have demonstrated a fundamental \emph{conservation law} for SMD and have used it to establish properties such as \emph{minimax optimality}, \emph{deterministic convergence}, and \emph{implicit regularization}  \cite{azizan2019stochastic,azizan2019characterization}. The main contribution of this paper is to provide a new stochastic interpretation of SMD, i.e., that it is \emph{risk-sensitive optimal}. This generalizes a similar result about SGD in the literature \cite{hassibi1999indefinite,hassibi1996h}. We also propose a new ``more symmetric" version of SMD, called  symmetric SMD (SSMD), which is suggested by our
analysis.

The paper is organized as follows. We review the main properties of SMD and the notion of Bregman divergence in Section~\ref{sec:background}. The risk-sensitive optimality result and its proof, as well as the new SSMD algorithm are provided is Section~\ref{sec:main}. We finally mention another stochastic result about SMD in Section~\ref{sec:other_stochastic}, and conclude in Section~\ref{sec:conclusion}.
\section{Background}\label{sec:background}

Consider a separable loss function of some unknown parameter (or weight) vector $w\in\mathbb{R}^p$:
\[ L(w) = \sum_{i=1}^nL_i(w), \]
where the $L_i(\cdot)$ are called the instantaneous (or local) loss functions, and where our goal is to minimize $L(\cdot)$ over $w$. For example, the conventional gradient descent (GD) algorithm can be used as an attempt to perform such minimization. A generalization of GD, called the {\em mirror descent} (MD) algorithm, was first introduced by Nemirovski and Yudin \cite{nemirovski1983problem} and can be described as follows. Consider a strictly convex differentiable function $\psi(\cdot)$, called the {\em potential function}. Then MD is given by the following recursion
\be
\nabla \psi(w_i) = \nabla\psi(w_{i-1})-\eta\nabla L(w_{i-1}),~~~w_0
\label{MD}
\ee
where $\eta>0$ is known as the step size or learning rate. Note that, due to the strict convexity of $\psi(\cdot)$, the gradient $\nabla\psi(\cdot)$ defines an invertible map so that the recursion in (\ref{MD}) yields a unique $w_i$ at each iteration. Compared to classical GD, rather than update the weight vector along the direction of the negative gradient, the update is done in the ``mirrored'' domain determined by the invertible transformation $\nabla\psi(\cdot)$. Mirror descent was originally conceived to exploit the geometrical structure of the problem by choosing an appropriate potential. Note that MD reduces to GD when $\psi(w) = \frac{1}{2}\|w\|^2$, since the gradient is simply the identity map. Other examples include the exponentiated gradient descent (aka the exponential weights) and the $p$-norms algorithm \cite{grove2001general,gentile2003robustness}. As with GD, it is straightforward to show that MD converges to a local minimum of $L(\cdot)$, provided the step size $\eta$ is small enough.

When $n$ is large, computation of the entire gradient may be cumbersome. Alternatively, in online scenarios, the entire loss function $L(\cdot)$ may not be available and only the local loss functions may be provided at each iteration. In such settings, a stochastic version of MD has been introduced, aptly called {\em stochastic mirror descent} (SMD), and which can be considered the straightforward generalization of stochastic gradient descent (SGD):
\be
\nabla \psi(w_i) = \nabla\psi(w_{i-1})-\eta\nabla L_i(w_{i-1}),~~~w_0
\label{SMD}
\ee
In the offline setting, the various instantaneous loss functions $L_i(\cdot)$ can either be drawn at random, or cycled through periodically. In the online setting, they are provided at each iteration. Unlike MD (and GD), for a fixed step size $\eta$, SMD does not generally converge, unless there exists a $w$ that {\em simultaneously} minimizes every local loss function $L_i(\cdot)$.\footnote{Since if this is not the case, even if the current estimate were at a local minimum of global loss function $L(\cdot)$, $w_*$, say, any of the local gradients $\nabla L_i(w_*)$ could be nonzero which would move us away from $w_*$.} For this reason, SMD with vanishing learning rate has also been considered
\be
\nabla \psi(w_i) = \nabla\psi(w_{i-1})-\eta_i\nabla L_i(w_{i-1}),~~~w_0
\label{MDvan}
\ee
where the learning rate is chosen such that $\eta_i\rightarrow 0$. With a vanishing learning rate it is not surprising that one can attain convergence (since after a while the algorithm is barely updating the weight vector)---what is more interesting is the fact that under suitably decaying rates one can obtain convergence to a local minimum of $L(\cdot)$ (more on this below). 

\subsection{Bregman Divergence}

For any given strictly convex differentiable potential function $\psi(\cdot)$, the {\em Bregman divergence} is defined as
\be
D_\psi(w,w') = \psi(w)-\psi(w')-\nabla\psi(w')^T(w-w') .
\ee
In other words, the Bregman divergence is the difference between the value of the function $\psi(\cdot)$ at a point $w$ and the value of its linear (or first order) approximation around another point $w'$ (see Fig.~\ref{fig:bregman}). Since a defining property of a convex function is that its linear approximations always lies below it, we have that $D_\psi(w,w')\geq 0$. Furthermore, since $\psi(\cdot)$ is strictly convex, we have that $D_\psi(w,w') = 0$ iff $w = w'$. Finally, it can be observed that $D_\psi(\cdot,\cdot)$ is convex in its first argument (but not necessarily in the second).

Since the Bregman divergence retains the quadratic (and higher order) terms in the error of the linear approximation of $\psi(w)$ around $w'$, it {\em inherits many of the properties of quadratics}. For example,  the classical ``law of cosines''
\[  \|w-w'\|^2 = \|w-w''\|^2+\|w''-w'\|^2-2(w'-w'')^T(w-w'') \]
generalizes to
\begin{multline}
D_\psi(w,w') = D_\psi(w,w'')+D_\psi(w'',w') \\
-\left(\nabla\psi(w')-\nabla\psi(w'')\right)^T(w-w'').
\label{cosine}
\end{multline}
More important for our developments is the following generalization of ``completion-of-squares'', which we formalize as a lemma.
\begin{lemma} Let $\psi_1(\cdot)$ and $\psi_2(\cdot)$ be strictly convex differentiable functions. Then it holds that
\begin{multline}
D_{\psi_1}(w,w_1)+ D_{\psi_2}(w,w_2) = D_{\psi_1}(w_*,w_1)+ D_{\psi_2}(w_*,w_2) \\
+D_{\psi_1+\psi_2}(w,w_*), 
\end{multline}
where $w_*$ is the unique solution to the equation
\be
\nabla(\psi_1+\psi_2)(w_*) = \nabla\psi_1(w_1)+\nabla\psi_2(w_2).
\ee
\label{lem:ccos}
\end{lemma}

\begin{proof}
The identities can be verified by straightforward calculation. The uniqueness of $w_*$ follows from the fact that $\psi_1(\cdot)+\psi_2(\cdot)$ is strictly convex since it is the sum of two such functions.
\end{proof}

\begin{figure}[btp]
\begin{center}
\includegraphics[width=0.55\columnwidth]{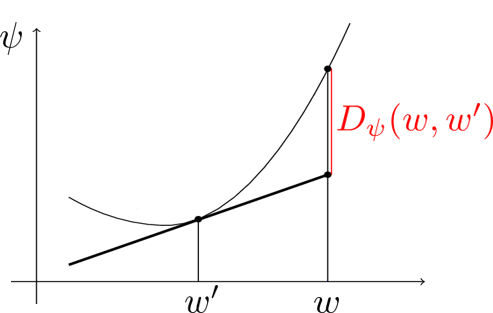}
\end{center}
\caption{Bregman divergence}
\label{fig:bregman}
\end{figure}
For example, if $\psi(w) = \|w\|^2$ then $D(w,w') = \|w-w'\|^2$, and if $\psi(p) = -H(p)$, where $p$ is a probability vector, then we get that $D_{-H}(p,p') = \sum_ip_i\log\frac{p_i}{p'_i}$ is the KL divergence (or relative entropy). 

The last fact about the Bregman divergence that we would like to mention is that a random variable $w$ that has a distribution $w\sim e^{-D_\psi(\cdot,w_0)}$ (i.e. $p(w)=ce^{-D_\psi(w,w_0)}$ for a suitable normalization constant $c$) is a member of the exponential family of distributions, and satisfies the property
\be
\mathbb{E}\nabla\psi(w)=\nabla\psi(w_0).
\ee
In other words, $w_0$ is the point whose mirror is the mean of the mirror map.

\subsection{Parametric Models}

It will now be useful to introduce some parametric models and make our loss functions more explicit. To this end, assume we have a collection of data points
\[ \left\{(x_i,y_i),i=1,\ldots n\right\} \]
where $x_i\in\mathbb{R}^m$ is the input and $y_i\in\mathbb{R}$ is the output. We will assume that the pairs $(x_i,y_i)$ are related through some parametric model
\be
y_i = f(x_i,w)+v_i,~~~~i=1,\ldots n
\label{nonlinmodel}
\ee
where $f(\cdot,\cdot)$ is a given function and represents the modeling class we are considering, $w\in\mathbb{R}^p$ is the unknown weight vector (or parameter), and $v_i$ represents both measurement noise and modeling errors. In this setting, the global loss function can be written as
\be
 L(w) = \sum_{i=1}^n\underbrace{\ell(y_i,f(x_i,w))}_{L_i(w)},
\ee
where $\ell(\cdot,\cdot)$ is a (differentiable) local loss function, with the property that $\ell(y_i,f(x_i,w)) = 0$ iff $y_i = f(x_i,w)$. Often $\ell(y_i,f(x_i,w)) = \ell(y_i-f(x_i,w))$, with $\ell(\cdot)$ convex and having a global minimum at zero. In this case,
\be
L(w) = \sum_{i=1}^n \ell(y_i-f(x_i,w)).
\label{global_loss}
\ee
For example, for quadratic loss we obtain $L(w) = \sum_{i=1}^n\frac{1}{2}(y_i-f(x_i,w))^2$. For (\ref{global_loss}), SMD takes the explicit form
\be
\small
\nabla\psi(w_i) = \nabla\psi(w_{i-1})+\eta \frac{\partial f(x_i,w_{i-1})}{\partial w}\ell'(y_i-f(x_i,w_{i-1})),~~~w_0.
\label{smd_mod}
\ee
An important special case is that of linear models
\be
y_i = x_i^Tw+v_i,~~~~i=1,\ldots, n
\label{linmodel}
\ee
where SMD takes the form
\be
\nabla\psi(w_i) = \nabla\psi(w_{i-1})+\eta x_i \ell'(y_i-x_i^Tw_{i-1}),~~~w_0.
\label{smd_linmod}
\ee

\subsection{Local and Global Interpretations of SMD}

It is straightforward to show that at each iteration, SMD solves the following optimization problem:
\be
w_i = \mbox{arg}\min_w ~D_\psi(w,w_{i-1})+\eta w^T\nabla L_i(w_{i-1}),
\label{local}
\ee
which can be verified by setting the gradient of the right hand side of (\ref{local}) to zero. What the above relation shows is that the SMD iterates try to align themselves with the direction of the instantaneous gradient, while also trying to stay close to the previous iterate in Bregman divergence. (The learning rate relatively weights these two objectives.) We refer to (\ref{local}) as the local interpretation of SMD. 

We have recently shown that SMD satisfies the following {\em local conservation law} \cite{azizan2019stochastic,azizan2019characterization}.
\begin{lemma}[Local Conservation Law \cite{azizan2019stochastic}] Even though the loss function $L_i(w) = \ell(y_i-f(x_i,w))$ may not be convex, define the Bregman divergence $D_{L_i}(w,w')$ in the usual way. Further define the quantity 
\be 
E_i(w_i,w_{i-1}) := D_{\psi-\eta L_i}(w_i,w_{i-1})+\eta L_i(w_i).
\ee
Then for each iteration of the SMD updates (\ref{smd_mod}), it holds that
\begin{multline}
D_\psi(w,w_{i-1})+\eta \ell(v_i) = D_\psi(w,w_i) \label{localcons} \\
+\eta D_{L_i}(w,w_{i-1})+E_i(w_i,w_{i-1}).
\end{multline}
\end{lemma}

Summing the local identities in (\ref{localcons}) from time 1 to time $T$ leads to the following global conservation law
\begin{multline}
D_\psi(w,w_0)+\eta\sum_{i=1}^T\ell(v_i) = D_\psi(w,w_T) \\ +\eta\sum_{i=1}^TD_{L_i}(w,w_{i-1})
+\sum_{i=1}^TE_i(w_i,w_{i-1})  \label{globalcons}
\end{multline}
Note that (\ref{globalcons}) holds for {\em any} horizon $T$. We refer to it as the global interpretation of SMD. It can be used to show several remarkable {\em deterministic} properties of the SMD algorithm. We now mention a couple.

\begin{figure}[thpb]
\centering
\includegraphics[width=\columnwidth]{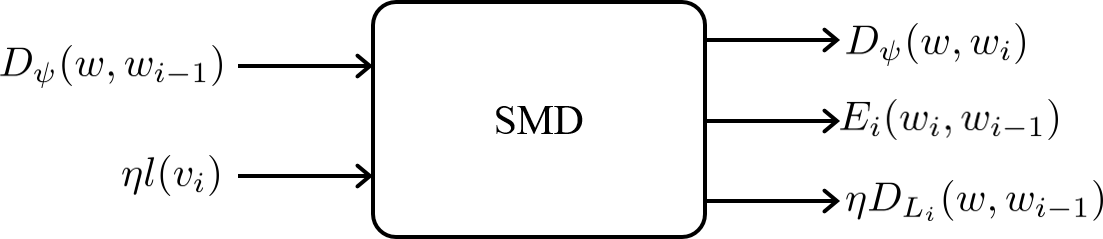}
\caption{Local Conservation Law of SMD}
\end{figure}

\subsection{Minimax Optimality of SMD}

Using the aforementioned global identity, in \cite{azizan2019stochastic,azizan2019characterization}, the following has been shown.
\begin{theorem}[Minimax Optimality \cite{azizan2019stochastic}]
For any $T$, provided $\eta$ is small enough so that $\psi(w)-\eta L_i(w)$ is convex for all $i$, then
\be
\min_{\{w_i\}}~\max_{w,\{v_i\}}~\frac{D_\psi(w,w_T)+\eta\sum_{i=1}^TD_{L_i}(w,w_{i-1})}{D_\psi(w,w_0)+\eta\sum_{i=1}^T \ell(v_i)} = 1
\ee
and SMD with learning rate $\eta$ is a minimax optimal algorithm achieving the above.
\label{thm:smd_minimix-optimal}
\end{theorem}
Theorem \ref{thm:smd_minimix-optimal} is a generalization of the $H^\infty$-optimality of the SGD algorithm for linear models and quadratic loss, where it is referred to as LMS \cite{hassibi1999indefinite,hassibi1996h,hassibi1994hoo}, to SMD and general models and general losses.
When the potential and loss are quadratic, we have $D_\psi(w,w_0) = \|w-w_0\|^2$ and $\ell(v_i) = v_i^2$. The quantity $D_{L_i}(w,w_{i-1}) = (y_i-x_i^Tw)^2-(y_i-x_i^Tw_{i-1})^2+2x_i^T(w-w_{i-1})(y_i-x_i^Tw_{i-1})$, after some simplification, takes on the form
\[ D_{L_i}(w,w_{i-1}) = (x_i^T(w-w_{i-1}))^2, \]
which is the square of the so-called {\em prediction error}. In this case, we recover the $H^\infty$-optimality of LMS, namely that it solves 
\be
\min_{\{w_i\}}~\max_{w,\{v_i\}}~\frac{\|w-w_T\|^2+\eta\sum_{i=1}^T(x_i^T(w-w_{i-1}))^2}{\|w-w_0\|^2+\eta\sum_{i=1}^Tv_i^2}
\ee
and the optimal value is $1$.
As mentioned above, Theorem \ref{thm:smd_minimix-optimal} generalizes $H^\infty$-optimality in three ways: it holds for general potential, general loss function, and general nonlinear model.

\subsection{Convergence and Implicit Regularization}
Another interesting property of SMD, which again can be proven using the global conservation law \eqref{globalcons}, is what is referred to as \emph{implicit regularization}. In over-parameterized (underdetermined) models, which are common in compressed sensing and modern deep learning problems, there are (typically a lot) more parameters (unknowns) than data points (measurements). That means there are many parameter vectors (in fact infinitely many) that are consistent with the observations:
$$\mathcal{W} = \left\{w\in\mathbb{R}^m\mid y_i=x_i^Tw,\ i=1,\dots,n\right\}.$$
The questions of interest in this regime are (1) does SMD converge to a solution? and (2) if it does so, which solution does it converge to?
The following result answers these questions.
\begin{theorem}[Convergence to the ``Closest'' Point\cite{azizan2019stochastic}]\label{thm:SMD_implicit-reg}
Suppose $l(\cdot)$ is differentiable and convex and has a unique root at $0$, $\psi(\cdot)$ is strictly convex, and $\eta>0$ is such that $\psi-\eta L_i$ is convex for all $i$. Then for any $w_0$, the SMD iterates converge to
\begin{equation}
w_{\infty}=\argmin_{w\in\mathcal{W}} D_{\psi}(w,w_0) .
\end{equation}
\end{theorem}
\begin{corollary}[Implicit Regularization\cite{azizan2019stochastic}]
In particular, for the initialization $w_0=\argmin_{w\in\mathbb{R}^m} \psi(w)$, under the conditions of Theorem~\ref{thm:SMD_implicit-reg}, the SMD iterates converge to
\begin{equation}\label{eq:minimum_potential}
w_{\infty}=\argmin_{w\in\mathcal{W}} \psi(w) .
\end{equation}
\end{corollary}
This means that running SMD, without any (explicit) regularization, results in a solution that has the smallest potential $\psi(\cdot)$ among all solutions, i.e., SMD implicitly regularizes the solution with $\psi(\cdot)$. In principle, one can choose the potential function  for {\em any} desired convex regularization. For example, we can find the maximum entropy solution by taking the potential to be the negative entropy, or do compressed sensing with $\psi(w) = \|w\|_{1+\epsilon}$ \cite{azizan2019stochastic,azizan2019characterization}.

We should remark that the result extends to quasi-convex losses $\ell(\cdot)$, and it holds locally (in an approximate sense) even for nonlinear models (non-convex cost).

\begin{figure}[tbp]
\begin{center}
\includegraphics[scale=.28]{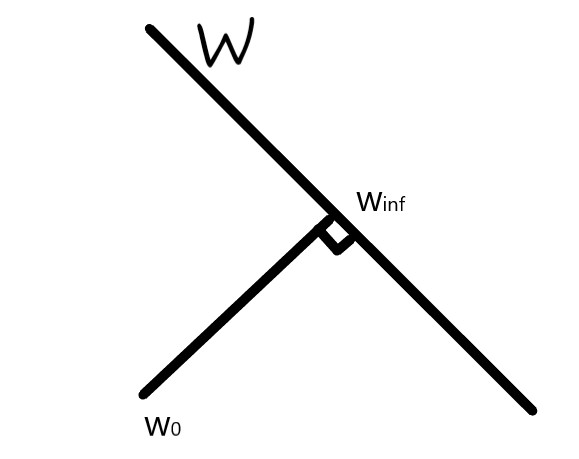}
\caption{$w_{\infty}$ is the closest solution (among all solutions $\mathcal{W}$) to $w_0$. Note that this picture is only for the Euclidean distance; in general the ``closest'' is measured in Bregman divergence.}
\end{center}
\end{figure}

\section{Main Results}\label{sec:main}

The results about SMD discussed in the previous section were deterministic. In this section, we give a stochastic interpretation of SMD, and show that it is risk-sensitive optimal.

\subsection{Risk-Sensitive Optimality of SMD}

Consider a stochastic model $y_i=x_i^Tw+v_i, i\geq 1$, where $w$ and $\{v_i\}$ are independent random variables with distributions $w\sim e^{-\frac{1}{\eta}D_\psi(\cdot,w_0)}$ and $v_i\sim e^{-\ell(\cdot)}$, which are members of the exponential family (note that when the potential function $\psi(\cdot)$ and the loss $\ell(\cdot)$ are square, both of these are Gaussian). A conventional quadratic estimator is one that minimizes the expected sum of squared prediction errors, i.e., 
\begin{equation}\label{min_quadratic}
\min_{\{z_i\}}\mathbb{E}_{|\{y_i\}}\left[\frac{1}{2}\sum_{i=1}^T(x_i^Tw-z_{i})^2\right] ,
\end{equation}
where the expectation is taken over $w$ and $\{v_i\}$ conditioned on the observations, and each $z_i$ in the minimization can only be a function of observations until time $i-1$.
For various problems, one may be interested in cost functions more general than quadratic, i.e.,
\begin{equation}\label{min_bregman}
\min_{\{z_i\}}\mathbb{E}_{|\{y_i\}}\left[\sum_{i=1}^TD_{\ell}(y_i-x_i^Tw,y_i-z_{i})\right] .
\end{equation}
The estimators that solve problems \eqref{min_quadratic} and \eqref{min_bregman} are referred to as ``risk-neutral'' estimators.

An alternative criterion is the ``risk-sensitive'' (or exponential cost) criterion, which was first introduced in \cite{jacobson1973optimal} and studied in \cite{speyer1974optimization,whittle1990risk,speyer1992optimal}. In particular, an estimator that solves the problem
\begin{equation}\label{min_exp_quadratic}
\min_{\{z_i\}}\mathbb{E}_{|\{y_i\}}\exp\left(\frac{1}{2}\sum_{i=1}^T(x_i^Tw-z_{i})^2\right) ,
\end{equation}
is called a ``risk-averse'' estimator. The reason is that in such a criterion, very large weights are placed on large errors, and hence, the estimator is more concerned about large values of error (their rare occurrence) than the moderate values of error.

Similar as in \eqref{min_bregman}, one can consider exponential cost of errors measured with a more general distance than quadratic, i.e.,
\begin{equation}\label{min_exp_bregman}
\min_{\{z_i\}}\mathbb{E}_{|\{y_i\}}\exp\left(\sum_{i=1}^TD_{\ell}(y_i-x_i^Tw,y_i-z_{i})\right) ,
\end{equation}

It has been shown in \cite{hassibi1996h,hassibi1999indefinite} that SGD for square loss (aka LMS) solves the problem \eqref{min_exp_quadratic}. In other words, LMS is risk-sensitive optimal. Formally, the result is as follows.
\begin{theorem}[Hassibi et al.\cite{hassibi1999indefinite}]\label{thm:SGD_risk-sensitive}
Consider the model $y_i=x_i^Tw+v_i, i\geq 1$, where $w$ and $\{v_i\}$ are independent Gaussian random variables with means $w_0$ and $0$ and variances $\eta I$ and $I$, respectively. Further, suppose that $\{x_i\}$ are persistently exciting and $0<\eta<\frac{1}{\|x_i\|^2}, \forall i$. Then the solution to the following optimization problem
$$\min_{\{z_i\}}\mathbb{E}_{|\{y_i\}}\exp\left(\frac{1}{2}\sum_{i=1}^T(x_i^Tw-z_i)^2\right)$$
where the expectation is taken over $w$ conditioned on the observations, and $z_i$ is only allowed to depend on observations up to time $i-1$, is given by $z_i=x_i^Tw_{i-1}$, where $\{w_i\}$ are the SGD iterates.
\end{theorem}
We should further remark that no larger exponent than $1/2$ is possible (no algorithm can attain a finite cost if the exponent is larger than $1/2$). 

The following result generalizes the risk-sensitive optimality of SGD for quadratic errors, to that of SMD for general Bregman-divergence errors.
\begin{theorem}\label{thm:SMD_risk-sensitive}
Consider the model $y_i=x_i^Tw+v_i, i\geq 1$, where $w$ and $\{v_i\}$ are independent random variables with distributions $w\sim e^{-\frac{1}{\eta}D_\psi(\cdot,w_0)}$ and $v_i\sim e^{-l(\cdot)}$. Further, suppose that $\{x_i\}$ are persistently exciting, and $\psi-\eta L_i$ is strictly convex for all $i$. Then the solution to the following optimization problem
\[ \min_{\{z_i\}} \mathbb{E}_{|\{y_i\}}\exp\left(\sum_{i=1}^TD_{\ell}(y_i-x_i^Tw,y_i-z_i)\right), \]
where the expectation is taken over $w$ conditioned on the observations, and $z_i$ is only allowed to depend on observations up to time $i-1$, is given by $z_i=x_i^Tw_{i-1}$, where $\{w_{i}\}$ are the SMD iterates.
\end{theorem}

\subsection{Proof of Theorem \ref{thm:SMD_risk-sensitive}}

The expected exponential cost that needs to be minimized in Theorem \ref{thm:SMD_risk-sensitive} is given by
\begin{multline*}
C\int\exp\left(-\frac{1}{\eta}D_\psi(w,w_0)-\sum_{i=1}^T\ell(y_i-x_i^Tw) \right. \\
\left. +\sum_{i=1}^TD_{\ell}(y_i-x_i^Tw,y_i-z_i)\right)dw,
\end{multline*}
where $C$ is a normalization constant that guarantees we are integrating
the cost against a conditional distribution. The challenge in evaluating the above integral over $w$ is that $w$ appears in all three terms of the exponent. In order to facilitate the computation of this integral, it will be useful to use the completion-of-squares formula of Lemma \ref{lem:ccos} to gather $w$ into a single term. The following lemma provides precisely what we need.

\begin{lemma} It holds that
{\small
\begin{multline*}
-\frac{1}{\eta}D_{\psi}(w,w_0)-\sum_{i=1}^T\ell(y_i-x_i^Tw)+\sum_{i=1}^TD_{\ell}(y_i-x_i^Tw,y_i-z_i) = \\
-\frac{1}{\eta}D_{\psi}(w,w_T)-\sum_{i=1}^T\left[\frac{1}{\eta}D_\psi(w_i,w_{i-1})+\ell(y_i-x_i^Tw_i)\right.\\
\left. -D_\ell(y_i-x_i^T,y_i-z_i)\right]
\end{multline*}
}
where the
$w_i$, $i=1,\ldots,T$ are given by the recursion
\be
\nabla\psi(w_i) = \nabla\psi(w_{i-1})+\eta x_i\ell'(y_i-z_i).
\label{gen_rec}
\ee
\label{lem:nice_iden}
\end{lemma} 
\begin{proof}
The proof is based on telescopically summing the local identity
{\small
\begin{multline*}
-\frac{1}{\eta}D_\psi(w,w_{i-1})-\ell(y_i-x_i^Tw)+D_\ell(y_i-x_i^Tw,y_i-z_i) = \\
-\frac{1}{\eta}D_\psi(w,w_{i})-\frac{1}{\eta}D_\psi(w_i,w_{i-1})\\
-\ell(y_i-x_i^Tw_i)+D_\ell(y_i-x_i^Tw_i,y_i-z_i),
\end{multline*}
}
from $i=1$ to $i=T$, where the $w_i$ are given through the recursion (\ref{gen_rec}). This local identity can be
either verified directly or obtained through two successive uses of Lemma~\ref{lem:ccos}.
\end{proof}

As promised, Lemma \ref{lem:nice_iden} gathers $w$ into a single term so that the integral over $w$ can be performed. Once this integral is performed, we are left with the following cost function
{\small
\begin{multline*}
C'\exp\left(-\sum_{i=1}^T\frac{1}{\eta}D_\psi(w_i,w_{i-1})+\ell(y_i-x_i^Tw_i)\right. \\
\left.-D_\ell(y_i-x_i^Tw_i,y_i-z_i)\right),
\end{multline*}
}
where $C'$ is a constant obtained after integrating out $w$. The above cost function must be recursively minimized over the $z_i$, which are only allowed to be functions of $\{y_j,j<i\}$, respectively. It is not clear how to do so from the above expression. The next lemma provides an identity that makes this recursive minimization straightforward. 
\begin{lemma} It holds that
{\small
\begin{multline*}
\ell(y_i-x_i^Tw_i)-D_\ell(y_i-x_i^Tw_i,y_i-z_i) = \\
\ell(y_i-x_i^Tw_{i-1})+\frac{1}{\eta}\left(\nabla\psi(w_i)-\nabla\psi(w_{i-1})
\right)^T(w_i-w_{i-1})\\
-D_\ell(y_i-x_i^Tw_{i-1},y_i-z_i).
\end{multline*}
}
\label{lem:nice_iden2}
\end{lemma}
\begin{proof}
This can be verified by perhaps tedious, but straightforward, calculations.
\end{proof}

In view of Lemma \ref{lem:nice_iden2}, the cost function to recursively minimize is 
{\small
\begin{multline*}
C'\exp\left(-\sum_{i=1}^T\frac{1}{\eta}D_\psi(w_i,w_{i-1})+\ell(y_i-x_i^Tw_{i-1})\right. \\
+\frac{1}{\eta}\left(\nabla\psi(w_i)-\nabla\psi(w_{i-1})
\right)^T(w_i-w_{i-1})\\
\left.-D_\ell(y_i-x_i^Tw_{i-1},y_i-z_i)\right).
\end{multline*}
}

Note that, at any time $i$, the only
term that $z_i$ has control over (in the sense that it is a term that depends only on past $y_j$) is the term
\[ D_\ell(y_i-x_i^Tw_{i-1},y_i-z_i). \]
(The other terms that are influenced by $z_i$, such as $w_i$, are influenced also by $y_i$---see (\ref{gen_rec})---so that $z_i$ cannot knowledgeably minimize them.) The term $D_\ell(y_i-x_i^Tw_{i-1},y_i-z_i)$ can be minimized, and in fact set to zero, by taking
\be
z_i = x_i^Tw_{i-1},
\ee
which when plugging into (\ref{gen_rec}) yields SMD. {\em This completes the proof.} 
(The attentive reader will have noticed that we needed Lemma \ref{lem:nice_iden2} since it was not clear how to minimize $D_\ell(y_i-x_i^Tw_i,y_i-z_i)$ over $z_i$, since we could not have taken $z_i = x_i^Tw_i$ as $w_i$ depends on $y_i$ and $z_i$ is not allowed to.)

\subsection{Symmetric SMD (SSMD)}

Our proof of the risk-sensitive optimality of SMD has led
 us to an alternative, and more symmetric version, of the algorithm that we refer to as symmetric SMD (or SSMD) and which may be of independent interest. The SSMD iterations are given by
\be
\nabla\psi (w_i) = \nabla\psi (w_{i-1})+\eta x_i \left(\ell'(y_i)-\ell'(x_i^Tw_{i-1})\right),~~~~w_0.
\label{SSMD}
\ee
SSMD satisfies the following
risk-sensitive optimality.
\begin{theorem}\label{thm:SSMD_risk-sensitive}
Consider the model $y_i=x_i^Tw+v_i, i\geq 1$, where $w$ and $\{v_i\}$ are independent random variables with $w|\{y_i\} \sim e^{-\frac{1}{\eta}D_\psi(\cdot,w_0)-D_\ell(x_i^T\cdot,y_i)}$. Further, suppose that $\{x_i\}$ are persistently exciting, and $\psi-\eta L_i$ is strictly convex for all $i$. Then the solution to the following optimization problem
\[ \min_{\{z_i\}} \mathbb{E}_{|\{y_i\}}\exp\left(\sum_{i=1}^TD_\ell(x_i^Tw,z_i)
\right), \]
where the expectation is taken over $w$ conditioned on the observations, and $z_i$ is only allowed to depend on observations up to time $i-1$, is given by $z_i=x_i^Tw_{i-1}$, where $\{w_{i}\}$ are the SSMD iterates.
\end{theorem}
\begin{proof}
The proof is similar to that of Theorem \ref{thm:SMD_risk-sensitive} and is omitted for brevity.
\end{proof}

We note that the difference between SMD and SSMD is that the
noise is now distributed according to $v_i\sim e^{-D_\ell(x_i^Tw,y_i)}$, rather than $v_i\sim e^{-\ell(y_i-x_i^Tw)}$, and that the exponent of the cost function is $D_\ell(x_i^Tw,z_i)$, rather than $D_\ell(y-x_i^Tw,y_i-z_i)$. The distributions and costs for SSMD appear to be more natural.  
\section{Other Stochastic Results}\label{sec:other_stochastic}
In the previous sections, we showed several fundamental deterministic and stochastic properties of SMD. One may ask how do these results relate to the conventional mean-square convergence results, such as \cite{nemirovski2009robust}. It turns out that the fundamental identity (conservation law \eqref{globalcons}) of SMD allows proving such stochastic convergence results in a direct way (which avoids appealing to stochastic differential equations and ergodic averaging) \cite{azizan2019characterization}. 

As mentioned before, for vanishing step size, convergence of any algorithm is not surprising, and is in fact trivial (because you are not updating anymore). However, the more interesting question is whether the algorithm converges to anything interesting. It turns out that when the data points are generated according to a stochastic model with white noise, SMD converges to the ``true'' parameter. More specifically, consider a model $y_i=x_i^Tw+v_i, i\geq 1,$ where $v_i$ are iid with $\E{v_i}=0$ and $\E{v_i^2}=\sigma^2$, and the inputs $x_i$ are ``persistently exciting,'' i.e., for any $\delta>0$, there exists $T>0$ s.t. $\sum_{i=1}^T x_ix_i^T\succeq\delta I$. Note that this is different from the setting of Theorem~\ref{thm:SMD_risk-sensitive}, in that the noises $v_i$ need not be Gaussian or from the the exponential family (the only assumption is whiteness), and the parameter $w$ is deterministic.
One can show that SMD with decaying step size indeed converges to $w$, under suitable conditions on the step size sequence.
\begin{theorem}

Consider the model $y_i=x_i^Tw+v_i, i\geq 1,$ where $\E{v_i}=0$, $\E{v_iv_j}=\sigma^2\delta_{ij}$, and the $x_i$ are persistently exciting. The stochastic mirror descent iterates for any strongly convex potential $\psi(\cdot)$, and a convex loss $\ell(\cdot)$ with a unique root at $0$, converge to $w$ in a mean-square sense, if the
the step size sequence $\{\eta_i\}$ satisfies $\sum_{i=1}^{\infty} \eta_i=\infty, \sum_{i=1}^{\infty} \eta_i^2<\infty$.
\end{theorem}
The step size conditions $\sum_{i=1}^{\infty} \eta_i=\infty, \sum_{i=1}^{\infty} \eta_i^2<\infty$ are known as Robbins--Monro \cite{robbins1951stochastic} conditions.

\section{Conclusion}\label{sec:conclusion}

In this paper, we reviewed several fundamental properties of stochastic mirror descent (SMD) family of algorithms, and provided a new stochastic interpretation of them, namely, that they are risk-sensitive optimal. The result generalizes a known result in the literature about the special case of SGD (aka LMS). Our analysis inspired a new algorithm, which is a ``more symmetric'' variant of SMD. Future work may concern studying this new algorithm and its convergence properties in more detail.




\bibliographystyle{IEEEtran}
\bibliography{references}

\begin{thebibliography}{10}
\providecommand{\url}[1]{#1}
\csname url@samestyle\endcsname
\providecommand{\newblock}{\relax}
\providecommand{\bibinfo}[2]{#2}
\providecommand{\BIBentrySTDinterwordspacing}{\spaceskip=0pt\relax}
\providecommand{\BIBentryALTinterwordstretchfactor}{4}
\providecommand{\BIBentryALTinterwordspacing}{\spaceskip=\fontdimen2\font plus
\BIBentryALTinterwordstretchfactor\fontdimen3\font minus
  \fontdimen4\font\relax}
\providecommand{\BIBforeignlanguage}[2]{{%
\expandafter\ifx\csname l@#1\endcsname\relax
\typeout{** WARNING: IEEEtran.bst: No hyphenation pattern has been}%
\typeout{** loaded for the language `#1'. Using the pattern for}%
\typeout{** the default language instead.}%
\else
\language=\csname l@#1\endcsname
\fi
#2}}
\providecommand{\BIBdecl}{\relax}
\BIBdecl

\bibitem{nemirovski1983problem}
A.~Nemirovski and D.~B. Yudin, ``Problem complexity and method efficiency in
  optimization.'' 1983.

\bibitem{beck2003mirror}
A.~Beck and M.~Teboulle, ``Mirror descent and nonlinear projected subgradient
  methods for convex optimization,'' \emph{Operations Research Letters},
  vol.~31, no.~3, pp. 167--175, 2003.

\bibitem{cesa2012mirror}
N.~Cesa-Bianchi, P.~Gaillard, G.~Lugosi, and G.~Stoltz, ``Mirror descent meets
  fixed share (and feels no regret),'' in \emph{Advances in Neural Information
  Processing Systems}, 2012, pp. 980--988.

\bibitem{zhou2017stochastic}
Z.~Zhou, P.~Mertikopoulos, N.~Bambos, S.~Boyd, and P.~W. Glynn, ``Stochastic
  mirror descent in variationally coherent optimization problems,'' in
  \emph{Advances in Neural Information Processing Systems}, 2017, pp.
  7043--7052.

\bibitem{nedic2014stochastic}
A.~Nedic and S.~Lee, ``On stochastic subgradient mirror-descent algorithm with
  weighted averaging,'' \emph{SIAM Journal on Optimization}, vol.~24, no.~1,
  pp. 84--107, 2014.

\bibitem{azizan2019characterization}
N.~Azizan and B.~Hassibi, ``A characterization of stochastic mirror descent
  algorithms and their convergence properties,'' in \emph{IEEE International
  Conference on Acoustics, Speech and Signal Processing (ICASSP)}, 2019.

\bibitem{raginsky2012continuous}
M.~Raginsky and J.~Bouvrie, ``Continuous-time stochastic mirror descent on a
  network: Variance reduction, consensus, convergence,'' in \emph{2012 IEEE
  51st IEEE Conference on Decision and Control (CDC)}.\hskip 1em plus 0.5em
  minus 0.4em\relax IEEE, 2012, pp. 6793--6800.

\bibitem{nemirovski2009robust}
A.~Nemirovski, A.~Juditsky, G.~Lan, and A.~Shapiro, ``Robust stochastic
  approximation approach to stochastic programming,'' \emph{SIAM Journal on
  optimization}, vol.~19, no.~4, pp. 1574--1609, 2009.

\bibitem{nesterov2009primal}
Y.~Nesterov, ``Primal-dual subgradient methods for convex problems,''
  \emph{Mathematical programming}, vol. 120, no.~1, pp. 221--259, 2009.

\bibitem{xiao2010dual}
L.~Xiao, ``Dual averaging methods for regularized stochastic learning and
  online optimization,'' \emph{Journal of Machine Learning Research}, vol.~11,
  no. Oct, pp. 2543--2596, 2010.

\bibitem{gunasekar2018characterizing}
S.~Gunasekar, J.~Lee, D.~Soudry, and N.~Srebro, ``Characterizing implicit bias
  in terms of optimization geometry,'' \emph{arXiv preprint arXiv:1802.08246},
  2018.

\bibitem{azizan2019stochastic}
N.~Azizan and B.~Hassibi, ``Stochastic gradient/mirror descent: Minimax
  optimality and implicit regularization,'' in \emph{International Conference
  on Learning Representations (ICLR)}, 2019.

\bibitem{hassibi1999indefinite}
B.~Hassibi, A.~H. Sayed, and T.~Kailath, \emph{Indefinite-Quadratic Estimation
  and Control: A Unified Approach to H2 and H-infinity Theories}.\hskip 1em
  plus 0.5em minus 0.4em\relax SIAM, 1999, vol.~16.

\bibitem{hassibi1996h}
------, ``Hoo optimality of the {LMS} algorithm,'' \emph{IEEE Transactions on
  Signal Processing}, vol.~44, no.~2, pp. 267--280, 1996.

\bibitem{grove2001general}
A.~J. Grove, N.~Littlestone, and D.~Schuurmans, ``General convergence results
  for linear discriminant updates,'' \emph{Machine Learning}, vol.~43, no.~3,
  pp. 173--210, 2001.

\bibitem{gentile2003robustness}
C.~Gentile, ``The robustness of the p-norm algorithms,'' \emph{Machine
  Learning}, vol.~53, no.~3, pp. 265--299, 2003.

\bibitem{hassibi1994hoo}
B.~Hassibi, A.~H. Sayed, and T.~Kailath, ``Hoo optimality criteria for {LMS}
  and backpropagation,'' in \emph{Advances in Neural Information Processing
  Systems 6}, 1994, pp. 351--358.

\bibitem{jacobson1973optimal}
D.~Jacobson, ``Optimal stochastic linear systems with exponential performance
  criteria and their relation to deterministic differential games,'' \emph{IEEE
  Transactions on Automatic control}, vol.~18, no.~2, pp. 124--131, 1973.

\bibitem{speyer1974optimization}
J.~Speyer, J.~Deyst, and D.~Jacobson, ``Optimization of stochastic linear
  systems with additive measurement and process noise using exponential
  performance criteria,'' \emph{IEEE Transactions on Automatic Control},
  vol.~19, no.~4, pp. 358--366, 1974.

\bibitem{whittle1990risk}
P.~Whittle, \emph{Risk-sensitive optimal control}.\hskip 1em plus 0.5em minus
  0.4em\relax John Wiley \& Son Ltd, 1990.

\bibitem{speyer1992optimal}
J.~L. Speyer, C.-H. Fan, and R.~N. Banavar, ``Optimal stochastic estimation
  with exponential cost criteria,'' in \emph{[1992] Proceedings of the 31st
  IEEE Conference on Decision and Control}.\hskip 1em plus 0.5em minus
  0.4em\relax IEEE, 1992, pp. 2293--2299.

\bibitem{robbins1951stochastic}
H.~Robbins and S.~Monro, ``A stochastic approximation method,'' \emph{The
  annals of mathematical statistics}, pp. 400--407, 1951.

\end{thebibliography}

\end{document}